\documentclass[11pt]{amsart}
\usepackage{amsmath,amssymb,amscd,amsthm,amsfonts,amstext,
hyperref,
mathtools,stmaryrd,enumitem,bbm,ifsym} 

\usepackage[textsize=footnotesize,color=blue!40, bordercolor=white]{todonotes}

\usepackage
[a4paper, margin=1.2in]{geometry}

\usepackage{MnSymbol} 
\usepackage{verbatim}

\newcommand{\D}{\mathcal{D}}

\newcommand{\cof}{\operatorname{cof}}

\newcommand{\On}{{\mathrm{Ord}}}

\newcommand{\GCH}{{\mathrm{GCH}}}

\newcommand{\ZFC}{{\sf ZFC}}
\newcommand{\ZF}{{\sf ZF}}

\newcommand{\dom}[1]{{{\rm{dom}}(#1)}}

\newcommand{\VV}{{\rm{V}}}

\newcommand{\Add}[2]{{\rm{Add}}({#1},{#2})}

\newtheorem{theorem}{Theorem}[section]
\newtheorem{lemma}[theorem]{Lemma}
\newtheorem{corollary}[theorem]{Corollary}
\newtheorem{proposition}[theorem]{Proposition}
\newtheorem{conjecture}[theorem]{Conjecture}
\newtheorem{question}[theorem]{Question}

\newtheorem{observation}[theorem]{Observation}
\newtheorem{claim}[theorem]{Claim}
\newtheorem*{claim*}{Claim}

\newtheorem*{subclaim*}{Subclaim}

\theoremstyle{definition}
\newtheorem{definition}[theorem]{Definition}
\newtheorem{to-do}[theorem]{Todo}

\theoremstyle{remark}
\newtheorem{remark}[theorem]{Remark}

\newenvironment{enumerate-(a)}{\begin{enumerate}[label={\upshape (\alph*)}, leftmargin=2pc]}{\end{enumerate}}

\newenvironment{enumerate-(a)-r}{\begin{enumerate}[label={\upshape (\alph*)}, leftmargin=2pc,resume]}{\end{enumerate}}

\newenvironment{enumerate-(A)}{\begin{enumerate}[label={\upshape (\Alph*)}, leftmargin=2pc]}{\end{enumerate}}

\newenvironment{enumerate-(A)-r}{\begin{enumerate}[label={\upshape (\Alph*)}, leftmargin=2pc,resume]}{\end{enumerate}}

\newenvironment{enumerate-(i)}{\begin{enumerate}[label={\upshape (\roman*)}, leftmargin=2pc]}{\end{enumerate}}

\newenvironment{enumerate-(i)-r}{\begin{enumerate}[label={\upshape (\roman*)}, leftmargin=2pc,resume]}{\end{enumerate}}

\newenvironment{enumerate-(I)}{\begin{enumerate}[label={\upshape (\Roman*)}, leftmargin=2pc]}{\end{enumerate}}

\newenvironment{enumerate-(I)-r}{\begin{enumerate}[label={\upshape (\Roman*)}, leftmargin=2pc,resume]}{\end{enumerate}}

\newenvironment{enumerate-(1)}{\begin{enumerate}[label={\upshape (\arabic*)}, leftmargin=2pc]}{\end{enumerate}}

\newenvironment{enumerate-(1)-r}{\begin{enumerate}[label={\upshape (\arabic*)}, leftmargin=2pc,resume]}{\end{enumerate}}

\begin{document}

\thanks{We would like to thank Joel Hamkins for a very helpful discussion on the subject matter, and we would like to thank Victoria Gitman for reading an early draft and pointing out several inaccuracies, which greatly helped us to improve this paper. The first author would also like to thank the Bristol logic group for giving him an opportunity to talk about an early version of this paper in their seminar, and for making helpful comments and remarks. We would also like to thank Victoria Gitman, Dan Nielsen and Philip Welch for additional insights into the subject matter. Finally, we would like to thank the anonymous referee for a large number of helpful remarks. Both authors were supported by DFG-grant LU2020/1-1.}



\author{Peter Holy}
\address{Peter Holy, Math. Institut, Universit\"at Bonn,
Endenicher Allee 60, 53115 Bonn, Germany}
\email{pholy@math.uni-bonn.de}




\author{Philipp Schlicht}
\address{Philipp Schlicht, Math. Institut, Universit\"at Bonn, 
Endenicher Allee 60, 53115 Bonn, Germany}
\email{schlicht@math.uni-bonn.de}

\date{\today}

\title{A hierarchy of Ramsey-like cardinals}

\begin{abstract}
We introduce a hierarchy of large cardinals between weakly compact and measurable cardinals, that is closely related to the Ramsey-like cardinals introduced by Victoria Gitman in \cite{MR2830415}, and is based on certain infinite \emph{filter games}, however also has a range of equivalent characterizations in terms of elementary embeddings. The aim of this paper is to locate the Ramsey-like cardinals studied by Gitman, and other well-known large cardinal notions, in this hierarchy. 
\end{abstract} 
\maketitle

\setcounter{tocdepth}{1}
\tableofcontents

\section{Introduction} 

Ramsey cardinals are a very popular and well-studied large cardinal concept in modern set theory. Like many, or perhaps most other large cardinal notions, they admit a characterization in terms of elementary embeddings, which is implicit in the work of William Mitchell (\cite{MR534574}), and explicitly isolated by Victoria Gitman in \cite[Theorem 1.3]{MR2830415} -- we provide the statement of this characterization in Theorem \ref{ramseyembeddings} below. However this embedding characterization does not lend itself very well to certain set theoretic arguments (for example, indestructibility arguments), as it is based on elementary embeddings between very weak structures. Therefore, Gitman considered various strengthenings of Ramsey cardinals in her \cite{MR2830415}, that she calls Ramsey-like cardinals, the definitions of which are based on the existence of certain elementary embeddings between stronger models of set theory -- we will review her definitions in Section \ref{gitman}.

In this paper, we want to introduce a whole hierarchy of Ramsey-like cardinals, that have a uniform definition, and, as we will show, are closely related to the Ramsey-like cardinals defined by Gitman, but which may be seen, as we will try to argue, to give rise to more natural large cardinal concepts than Gitman's Ramsey-like cardinals.

We will also show that the Ramsey-like cardinals in our hierarchy are very robust in the sense that they have a range of equivalent characterizations, in particular one that is based on certain infinite games on regular and uncountable cardinals $\kappa$, where one of the players provides $\kappa$-models, and the other player has to measure the subsets of $\kappa$ appearing in those models in a coherent way. These games will be introduced in Section \ref{filtergames}. They are what actually led us to the discovery of our hierarchy of Ramsey-like cardinals, and they may also be of independent interest.

Our new hierarchy of Ramsey-like cardinal will then be introduced and studied in some detail in Section \ref{ramseylikehierarchy}. We will also study the closely related concept of \emph{filter sequences} in Section \ref{section:filtersequences}. While most large cardinals in our new hierarchy are strengthenings of Ramseyness, in Section \ref{section:absoluteness}, we will show that the weakest large cardinal concept in our new hierarchy is downwards absolute to the constructible universe $L$. In Section \ref{section:measurability}, we show that one of the strongest concepts in our new hierarchy can consistently be separated from measurability. We provide some questions in Section \ref{questions}, and we also provide some very recent answers by Victoria Gitman, Dan Nielsen and Philip Welch to some of these questions in the final Section \ref{answers}.

\section{Strengthenings of the filter property}\label{basic}

In this section, we will consider some natural attempts at strengthening the filter property (the statement of which is found in Definition \ref{definition:filterproperty} below) of weakly compact cardinals, most of which however will turn out to either be inconsistent or fairly weak. This will motivate the definition of the $\gamma$-filter properties, a hierarchy of strengthenings of the filter property, that lies in the range between ineffable and measurable cardinals, in Section \ref{filtergames}. We will start by introducing a slightly generalized notion of filter, which will be useful in several places. Before we actually do so, we also need to introduce our notion of (weak) $\kappa$-model. Unlike usual, we do not require those to be transitive.

\begin{definition}
  A \emph{weak $\kappa$-model} is a set $M$ of size $\kappa$ with $\kappa+1\subseteq M$ and such that  $\langle M,\in\rangle\models\ZFC^-$, that is $\ZFC$ without the powerset axiom (but, as is usual, with the scheme of collection rather than replacement). A weak $\kappa$-model is a \emph{$\kappa$-model} if additionally $M^{<\kappa}\subseteq M$.
\end{definition}

Since we will consider filters over subsets of $\mathcal{P}(\kappa)$, where $\kappa$ is a cardinal, we use the following modified definitions of filters (one could also call these \emph{partial filters}, but we would like to stick to the notion of filter also for the generalized versions below). 

\begin{definition}\label{definition:filter}
\begin{enumerate-(a)} 
\item
A \emph{filter} on $\kappa$ is a subset $F$ of $\mathcal P(\kappa)$ such that $|\bigcap_{i<n}A_i|=\kappa$ whenever $n\in\omega$ and $\langle A_i\mid i<n\rangle$ is a sequence of elements of $F$.\footnote{In particular, this implies that every element of a filter on $\kappa$ 
has size $\kappa$. 
}
\item 
A filter $F$ on $\kappa$ \emph{measures} a subset $A$ of $\kappa$ if $A\in F$ or $\kappa\setminus A\in F$. 
$F$ 
\emph{measures} a subset $X$ of $\mathcal P(\kappa)$ if $F$ measures every element of $X$. $F$ is an \emph{ultrafilter} on $\kappa$ if it measures $\mathcal P(\kappa)$. 
\item 
A filter $F$ on $\kappa$ is \emph{${<}\kappa$-complete} if $|\bigcap_{i<\gamma}X_i|=\kappa$ 
for every sequence $\langle X_i\mid i<\gamma\rangle$ with $\gamma<\kappa$ and $X_i\in F$ for all $i<\gamma$.
\item
If $\kappa$ is regular, a filter $F$ on $\kappa$ is \emph{normal} if for every sequence $\vec{X}=\langle X_\alpha\mid \alpha<\kappa\rangle$ of elements of $F$, the diagonal intersection $\largetriangleup \vec{X}$ is a stationary subset of $\kappa$.
\item
If $M$ is a weak $\kappa$-model, then a filter $F$ on $\kappa$ is \emph{$M$-normal} if it measures $\mathcal P(\kappa)\cap M$ and $\largetriangleup \vec{X}\in F$ whenever $\vec X=\langle X_\alpha\mid\alpha<\kappa\rangle\in M$ is a sequence of elements of $F$.
\end{enumerate-(a)} 
\end{definition} 

\begin{definition} \label{definition:filterproperty}
Suppose that $\kappa$ is an uncountable cardinal. $\kappa$ has the \emph{filter property} if for every subset $X$ of $P(\kappa)$ of size ${\leq}\kappa$, there is a ${<}\kappa$-complete filter $F$ on $\kappa$ which measures $X$. 
\end{definition} 


It is well-known (see \cite[Theorem 1.1.3]{MR0460120}) that an uncountable cardinal $\kappa$ has the filter property if and only if $\kappa$ is weakly compact. 
If $\vec{X}=\langle X_\alpha\mid \alpha<\kappa\rangle$ is a sequence, we write $\largetriangleup \vec{X}=\largetriangleup_{\alpha<\kappa}X_\alpha$ for its diagonal intersection. 
Note that every normal filter on $\kappa$ is easily seen to be ${<}\kappa$-complete and to only contain stationary subsets of $\kappa$. If $F$ is a normal filter on $\kappa$ and $\vec X=\langle X_\alpha\mid\alpha<\kappa\rangle$ is a sequence of elements of $F$, then $\largetriangleup\vec X\in F$ whenever $F$ measures $\largetriangleup\vec X$. In particular, if a filter $F$ is normal and measures $\mathcal P(\kappa)\cap M$, then $F$ is $M$-normal. 
The reason for demanding that $\largetriangleup \vec{X}$ be stationary in Definition \ref{definition:filter}, (d) is provided by the next observation.

\begin{observation} \label{nonstationarydoesntwork}
Suppose that $F$ is a filter and $\vec X=\langle X_\alpha\mid\alpha<\kappa\rangle$ is a sequence of elements of $F$ such that $\largetriangleup\vec X$ is non-stationary. Then there is a subset $\D$ of $\mathcal{P}(\kappa)$ of size $\kappa$, such that every filter that extends $F$ and measures $\D$, contains a sequence $\vec Y=\langle Y_\alpha\mid\alpha<\kappa\rangle$, such that $\largetriangleup\vec Y=\emptyset$. 
\end{observation} 
\begin{proof} 
Suppose that $\vec X=\langle X_\alpha\mid\alpha<\kappa\rangle$ is a sequence of elements of $F$ and $\largetriangleup\vec X$ is nonstationary. Suppose that $C$ is a club subset of $\kappa$ that is disjoint from $\largetriangleup\vec X$. We consider the regressive function $f\colon \largetriangleup\vec X\rightarrow \kappa$ defined by $f(\alpha)=\max(C\cap \alpha)$ for $\alpha\in\largetriangleup\vec X$. Moreover, we consider the sequence $\vec{A}=\langle A_\alpha\mid \alpha<\kappa\rangle$ of bounded subsets $A_\alpha=f^{-1}[\{\alpha\}]$ of $\kappa$ for $\alpha<\kappa$.

Let $\D$ denote the closure under finite intersections and relative complements in $\kappa$ of the set consisting of the elements of $F$, $\largetriangleup\vec X$, the sets $A_\alpha$ for $\alpha<\kappa$ and of $\largetriangleup\vec A$. Suppose that $\bar F\subseteq\D$ extends $F$ and measures $\D$. Note that this implies that $\bar F$ is closed under finite intersections.

Suppose first that $\kappa\setminus\largetriangleup\vec X\in\bar F$. For every $\alpha<\kappa$, let $Y_\alpha=X_\alpha\setminus\largetriangleup\vec X\in\bar F$ and let $\vec Y=\langle Y_\alpha\mid\alpha<\kappa\rangle$. Then $\largetriangleup\vec Y=\emptyset$.

Now suppose that $\largetriangleup\vec X\in\bar F$. Since each $A_\alpha$ is a bounded subset of $\kappa$, $\kappa\setminus A_\alpha\in\bar F$ for every $\alpha<\kappa$. But then $\largetriangleup_{\alpha<\kappa}(\kappa\setminus A_\alpha)=\{\beta<\kappa\mid\beta\in\bigcap_{\gamma<\beta}(\kappa\setminus f^{-1}(\{\gamma\}))\}=\{\beta<\kappa\mid f(\beta)\ge\beta\,\lor\,\beta\not\in\dom f\}=\kappa\setminus\largetriangleup\vec X\not\in\bar F$. Making use of the sequence $\langle\kappa\setminus A_\alpha\mid\alpha<\kappa\rangle$ rather than $\vec X$, we are in the situation of the first case above, thus obtaining an empty diagonal intersection of elements of $\bar F$.
\end{proof}


A first attempt at strengthening the filter property is to require normality, and this will lead us from weak compactness to ineffability.

\begin{definition} 
An uncountable cardinal $\kappa$ has the \emph{normal filter property} if for every subset $X$ of $P(\kappa)$ of size $\leq \kappa$, there is a normal filter $F$ on $\kappa$ measuring $X$. It has the \emph{$M$-normal filter property} if there exists an $M$-normal filter on $\kappa$ for every weak $\kappa$-model $M$.
\end{definition} 

\begin{lemma} \label{characterization of normal filters}
Suppose that $F$ is a filter on $\kappa$ of size $\kappa$ and that $\vec{X}=\langle X_\alpha\mid \alpha<\kappa\rangle$ is an enumeration of $F$. Then $F$ is normal if and only if $\largetriangleup \vec{X}$ is stationary. 
\end{lemma} 
\begin{proof} 
Suppose that $\largetriangleup \vec{X}$ is stationary. Moreover, suppose that $\vec{Y}=\langle Y_\alpha\mid\alpha<\kappa\rangle$ and $g\colon\kappa\rightarrow\kappa$ is a function with $Y_\alpha=X_{g(\alpha)}$ for all $\alpha<\kappa$. 
Let $C_g=\{\alpha<\kappa\mid g[\alpha]\subseteq\alpha\}$ denote the club of closure points of $g$. Then 
\[\largetriangleup \vec{X}\cap C_g\subseteq \largetriangleup \vec{Y}\cap C_g\] 
and hence $\largetriangleup \vec{Y}$ is stationary. 
\end{proof} 



It is immediate from the embedding characterization of weakly compact cardinals, that weak compactness implies the $M$-normal filter property. On the other hand, if $\kappa^{<\kappa}=\kappa$, every $\kappa$-sized subset of $\mathcal P(\kappa)$ is contained, as a subset, in some $\kappa$-model $M$. Thus if the $M$-normal filter property holds for $\kappa=\kappa^{<\kappa}$, then $\kappa$ is weakly compact, as follows immediately from the filter property characterization of weakly compact cardinals. For the normal filter property, the following is an immediate consequence of \cite[Theorem 1]{MR594323} together with Lemma \ref{characterization of normal filters}. Remember that a cardinal $\kappa$ is \emph{ineffable} if whenever $\langle A_\alpha\mid\alpha<\kappa\rangle$ is a $\kappa$-list, that is $A_\alpha\subseteq\alpha$ for every $\alpha<\kappa$, then there is $A\subseteq\kappa$ such that $\{\alpha<\kappa\mid A\cap\alpha=A_\alpha\}$ is stationary.

\begin{proposition}[Di Prisco, Zwicker] 
An uncountable cardinal $\kappa$ has the normal filter property if and only if it is ineffable. \hfill{$\Box$}
\end{proposition}



We now want to turn our attention to natural attempts at strengthening the above filter properties, which are the following \emph{filter extension properties}. They will however turn out to either be trivial or inconsistent, and this will then lead us to a more successful attempt at strengthening the filter property in Section \ref{filtergames}.

\begin{definition} 
A cardinal $\kappa$ has the \emph{filter extension property} if for every ${<}\kappa$-complete filter $F$ on $\kappa$ of size at most $\kappa$ and for every subset $X$ of $\mathcal P(\kappa)$ of size at most $\kappa$, there is a ${<}\kappa$-complete filter $\bar{F}$ with $F\subseteq\bar{F}$ that measures $X$. 

A cardinal $\kappa$ that satisfies the filter property has the \emph{$M$-normal filter extension property} if for every weak $\kappa$-model $M$, every $M$-normal filter $F$ on $\kappa$ and every weak $\kappa$-model $N\supseteq M$, there is an $N$-normal filter $\bar F$ with $F\subseteq\bar F$.

$\kappa$ has the \emph{normal filter extension property} if for every normal filter $F$ on $\kappa$ of size at most $\kappa$ and every $X\subseteq\mathcal P(\kappa)$ of size at most $ \kappa$, there is a normal filter $\bar F\supseteq F$ that measures $X$.
\end{definition} 

\begin{proposition} \label{filter extension property for weakly compacts} 
Every weakly compact cardinal $\kappa$ satisfies the filter extension property. 
\end{proposition} 
\begin{proof} 
Let $F$ be a ${<}\kappa$-complete filter on $\kappa$ of size at most $\kappa$ and let $X$ be a subset of $\mathcal P(\kappa)$ of size at most $\kappa$. We construct a subtree $T$ of ${}^{{<}\kappa}2$ as follows. Suppose that $\langle A_i\mid i<\kappa\rangle$ is an enumeration of $F$ and $\langle B_i\mid i<\kappa\rangle$ is an enumeration of $X$. 

We define $\mathrm{Lev}_\alpha(T)$ for $\alpha<\kappa$ as follows. Let $B_{i,j}=B_i$ for $j=0$ and $B_{i,j}=\kappa\setminus B_i$ for $j=1$, where $i<\kappa$. If $t\in 2^\alpha$, let $A_\alpha=\bigcap_{i<\alpha}A_i$, let $B_{\alpha,t}=\bigcap_{i<\alpha} B_{i,t(i)}$ and let $t\in \mathrm{Lev}_\alpha(T)$ if $|A_\alpha\cap B_{\alpha,t}|=\kappa$. 
Then $T$ is a subtree of $2^{<\kappa}$. 

Since $|A_\alpha|=\kappa$ and $\langle B_{\alpha,t}\mid t\in 2^\alpha\rangle$ is a partition of $\kappa$, $\mathrm{Lev}_\alpha(T)\neq\emptyset$. 
Since $\kappa$ has the tree property, there is a cofinal branch $b$ through $T$. Let $\bar{F}=\{A\subseteq \kappa\mid \exists \alpha<\kappa\ A_\alpha\cap B_{\alpha,b\upharpoonright \alpha}\subseteq A\}$. 
Then $\bar{F}$ is a ${<}\kappa$-complete filter that measures $X$ and extends $F$. 
\end{proof} 

\begin{proposition} \label{normal filter extension property false}
The normal filter extension property fails for every infinite cardinal. 
\end{proposition} 
\begin{proof} 
The property clearly fails for $\omega$. Suppose for a contradiction that the normal filter extension property holds for some uncountable cardinal $\kappa$. Since this implies that the filter property holds for $\kappa$, we know that $\kappa$ is weakly compact. Suppose that $S=S^\kappa_\omega$ and that $F_0=\{S\}$. $F_0$ is a normal filter. Let $M$ be a $\kappa$-model with $S\in M$. Assume that $F_1$ is a normal filter on $\kappa$ that measures $\mathcal P(\kappa)\cap M$. Normality of $F_1$ easily implies that $F_1$ is $M$-normal and that the ultrapower $N$ of $M$ by $F_1$ is well-founded. By {\L}os' theorem, since $\kappa$ is represented by the identity function in $N$, $\kappa$ has cofinality $\omega$ in $N$, contradicting that $\kappa$ is inaccessible.
\end{proof} 

The counterexample of a normal filter that cannot be extended to a larger set in the above is somewhat pathological, and perhaps the more interesting question is whether the $M$-normal filter extension property is consistent for some (weakly compact) cardinal $\kappa$. This has recently been answered by Victoria Gitman, and we would like to thank her for letting us include her proof here. 
Before we can provide Gitman's proof this of result, we need to introduce some standard terminology, which will also be useful for the later sections of our paper. 

\begin{definition} 
  Suppose that $M$ is a weak $\kappa$-model. 
\begin{enumerate-(a)} 
\item 
An embedding $j\colon M\to N$ is \emph{$\kappa$-powerset preserving} if it has critical point $\kappa$ and $M$ and $N$ have the same subsets of $\kappa$.

\item 
An $M$-normal filter $U$ on $\kappa$ is \emph{weakly amenable} if for every $A\in M$ of size at most $\kappa$ in $M$, the intersection $U\cap A$ is an element of $M$. 
\item 
An $M$-normal filter $U$ on $\kappa$ is \emph{good} if it is weakly amenable and the ultrapower of $M$ by $U$ is well-founded. 
\end{enumerate-(a)} 
\end{definition}

We will often make use of the following lemma, that is provided in \cite[Section 19]{MR2731169} for transitive weak $\kappa$-models, however the same proofs go through for possibly non-transitive weak $\kappa$-models.

\begin{lemma}\label{wapp_equivalence}
  Suppose that $M$ is a weak $\kappa$-model.
  \begin{enumerate-(1)}
    \item If $j\colon M\to N$ is the well-founded ultrapower map that is induced by a weakly amenable $M$-normal filter on $\kappa$, then $j$ is $\kappa$-powerset preserving.
    \item If $j\colon M\to N$ is a $\kappa$-powerset preserving embedding, then the $M$-normal filter $U=\{A\in\mathcal P(\kappa)^M\mid\kappa\in j(A)\}$ is weakly amenable and induces a well-founded ultrapower of $M$.
  \end{enumerate-(1)}
\end{lemma}

\begin{proposition}[Gitman]
  The $M$-normal filter extension property fails at every (weakly compact) cardinal.
\end{proposition}
\begin{proof}
  Assume that $\kappa$ is the least weakly compact cardinal that satisfies the $M$-normal filter extension property. Observe first that if $M$ is any weak $\kappa$-model and $U$ is an $M$-normal filter on $\kappa$, then $U$ has to be countably complete, for if not $U$ cannot be extended to an $N$-normal filter for any $N\supseteq M$ containing a witness for $U$ not being countably complete. Let $M_0\prec H(\kappa^+)$ be a weak $\kappa$-model containing $V_\kappa$, and let $U_0$ be an $M_0$-normal filter. Given $M_i$ and $U_i$, let $M_{i+1}\prec H(\kappa^+)$ be a weak $\kappa$-model containing $M_i$ and $U_i$ as elements, and let $U_{i+1}$ be an $M_i$-normal filter extending $U_i$, making use of the $M$-normal filter extension property. Continue this construction for $\omega$ steps, let $M$ be the union of the $M_i$ and let $U$ be the union of the $U_i$. By construction, $U$ is weakly amenable for $M$, and by our above observation, we may assume that $U$ is countably complete. Let $j\colon M\to N$ be the ultrapower embedding induced by $U$. Now $M\prec H(\kappa^+)$ satisfies that $\kappa$ is weakly compact and has the $M$-normal filter extension property. But since $j$ is $\kappa$-powerset preserving, this is also true in $N$, and hence by elementarity, $\kappa$ cannot be least with this property.
\end{proof}

Having observed that both the $M$-normal and the normal filter extension property are inconsistent, the fact that the filter extension property is no stronger than the filter property might lead one to try and further strengthen the filter extension property in order to obtain something interesting. The filter extension property at $\kappa$ is equivalent to the second player winning the following finite game. Player I plays a ${<}\kappa$-complete filter $F_0$ on $\kappa$ and a collection $X$ of subsets of $\kappa$ of size $\kappa$. Player II wins if she can play a ${<}\kappa$-complete filter on $\kappa$ that extends $F_0$ and measures $X$. It is natural to investigate what happens if this game is continued into the transfinite.

Consider the following infinite two player game $G(\kappa)$ of perfect information. Two players, I and II, take turns to play a $\subseteq$-increasing sequence $\langle F_i\mid i<\omega\rangle$ of ${<}\kappa$-complete filters on $\kappa$ of size $\kappa$. Player II wins in case the filter $\bigcup_{i<\omega}F_i$ is ${<}\kappa$-complete.

One could define a variant of the filter (extension) property at $\kappa$ by requiring that Player I does not have a winning strategy in the game $G(\kappa)$. Note that however, as Joel Hamkins pointed out to us, this property is again inconsistent, that is Player I provably has a winning strategy in the game $G(\kappa)$. \footnote{Unlike in the finite game described above, Player I does not play subsets of $\mathcal P(\kappa)$ corresponding to the set $X$. Extending the game $G(\kappa)$ in this way would however make it even easier for Player I to win.} This result is essentially due to Jozef Schreier.

\begin{proposition}[Schreier]
  Let $\kappa$ be an uncountable cardinal. Then Play\-er I has a winning strategy in the game $G(\kappa)$.
\end{proposition}
\begin{proof} 
We first claim that instead of the game $G(\kappa)$, we can equivalently consider the game $G$ of length $\omega$, in which both players take turns to play a decreasing sequence of subsets of $\kappa$ of size $\kappa$, with the winning condition for Player II being that the intersection of those subsets has size $\kappa$. 
To see this, we translate a ${<}\kappa$-complete filter $F$ of size $\kappa$ to a subset $X$ of $\kappa$ as follows. 
Assuming that $\langle X_\alpha\mid \alpha<\kappa\rangle$ is an enumeration of $F$, we define a strictly increasing sequence $\langle x_\alpha\mid \alpha<\kappa\rangle$ of ordinals by choosing the least $x_\alpha\in\bigcap_{\beta<\alpha}X_\beta$ above the previous $x_\beta$ with $\beta<\alpha$, for each $0<\alpha<\kappa$, and then we let $X=\{x_\alpha\mid\alpha<\kappa\}$. On the other hand, given an unbounded set $X\subseteq\kappa$, we may define a ${<}\kappa$-complete filter $F$ by setting $F=\{X\setminus\alpha\mid\alpha<\kappa\}$. It is straightforward to verify that $X$ and $F$ can be used interchangeably, and this in particular allows us to translate strategies between $G(\kappa)$ and $G$. 
Schreier (\cite{schreier}) proved that Player I has a winning strategy $\sigma$ for $G$. This strategy is defined as follows. In each successor step, Player I enumerates the set previously played by Player II, and removes the least element in each $\omega$-block of the enumeration. An easy argument using the well-foundedness of the $\in$-relation on the ordinals shows $\sigma$ to be winning for Player I, with the intersection of the subsets of $\kappa$ played during a run of the game in which Player I plays according to $\sigma$ ending up as the empty set: If some ordinal $\alpha$ would lie in their intersection, then its position in the increasing enumeration of the individual subsets of $\kappa$ played during that run would strictly decrease after each move of Player I, giving rise to a strictly decreasing $\omega$-sequence of ordinals, which is a contradiction.
\end{proof}

Many further \emph{infinite filter games} can be defined. For example, if in the game $G(\kappa)$ above, we require all filters to be normal, we obtain a game for which the non-existence of a winning strategy for Player I implies the nonstationary ideal to be precipitous, for the modified game corresponds to the variant of the game $G$ where both players have to play stationary subsets of $\kappa$, with the winning condition for Player II being that the intersection of the stationary subsets is stationary, using that normal filters correspond to stationary sets via their diagonal intersection. It is well-known (see e.g.\ \cite[Lemma 22.21]{MR1940513}) that the precipitousness of the non-stationary ideal can be characterized by the non-existence of a winning strategy for Player I in the same game, however with the winning condition for Player II being a nonempty (and not necessarily stationary) intersection. 



\section{Filter games}\label{filtergames}

In this section, we want to investigate another way of strengthening the filter property at $\kappa$, by viewing it as being equivalent to the non-existence of a winning strategy for Player I in the following simple game of length $2$. Player I starts by playing a subset $X$ of $\mathcal P(\kappa)$ of size at most $\kappa$, and in order to win, Player II has to play a ${<}\kappa$-complete filter that measures $X$. It is again tempting to let this game (and variations of it) continue to greater (and in particular infinite) lengths, that is to have Player I (the \emph{challenger}) play increasingly larger subcollections of $\mathcal P(\kappa)$ of size at most $\kappa$, and to ask for Player II (the \emph{judge}) to measure them by increasingly larger ${<}\kappa$-complete filters in order to win. There are many variations in formalizing the details of such a game, and we will pick one particular such formalization in the following, the choice of which will be justified by its usefulness in the remainder of this paper.

\begin{definition}\label{filtergamesdefinition}
Given an ordinal $\gamma\le\kappa^+$ and regular uncountable cardinals $\kappa=\kappa^{<\kappa}<\theta$, consider the following two-player game of perfect information $G_\gamma^\theta(\kappa)$. Two players, the \emph{challenger} and the \emph{judge}, take turns to play $\subseteq$-increasing sequences $\langle M_\alpha\mid\alpha<\gamma\rangle$ of $\kappa$-models, and $\langle F_\alpha\mid\alpha<\gamma\rangle$ of filters on $\kappa$, such that the following hold for every $\alpha<\gamma$.
\begin{enumerate-(a)} 
  \item At any stage $\alpha<\gamma$, the challenger plays $M_\alpha$, and then the judge plays $F_\alpha$.
  \item $M_\alpha\prec H(\theta)$,
  \item $\langle M_{\bar\alpha}\mid\bar\alpha<\alpha\rangle, \langle F_{\bar\alpha}\cap M_{\bar\alpha}\mid\bar\alpha<\alpha\rangle\in M_\alpha$,
  \item $F_{\alpha}$ is a filter on $\kappa$ that measures $\mathcal P(\kappa)\cap M_\alpha$ and
  \item $F_\alpha\supseteq\bigcup_{\beta<\alpha}F_\beta$.
\end{enumerate-(a)}
Let $M_\gamma:=\bigcup_{\alpha<\gamma}M_\alpha$, and let $F_\gamma:=\bigcup_{\alpha<\gamma}F_\alpha$.
 If $F_\gamma$ is an $M_\gamma$-normal filter
, then the judge wins. Otherwise, the challenger wins. \footnote{The following possible alternative definition of the games $G_\gamma^\theta(\kappa)$ was remarked by Joel Hamkins, and provides a very useful perspective. In each step $\alpha<\gamma$, in order to have a chance of winning, the judge has to play not only an $M_\alpha$-normal filter $F_\alpha$, but in fact has to play some $F_\alpha$ which is normal, as follows by Observation \ref{nonstationarydoesntwork}. Thus by Lemma \ref{characterization of normal filters}, one might assume that rather than playing filters, the judge is just playing stationary sets which correspond to diagonal intersections of enumerations of the relevant filters.}
\end{definition}

We also define the following variation of the above games. For $\gamma$, $\kappa$ and $\theta$ as above, let $\overline{G_\gamma^\theta}(\kappa)$ denote the variant of $G_\gamma^\theta(\kappa)$ where we additionally require the judge to play such that each $F_\alpha\subseteq M_\alpha$, that is she is not allowed to measure more sets than those in $M_\alpha$ in her $\alpha^\textrm{th}$ move, for every $\alpha<\gamma$.

\begin{lemma}\label{barequivalence}
  Let $\gamma\le\kappa^+$, let $\kappa=\kappa^{<\kappa}$ be an uncountable cardinal, and let $\theta>\kappa$ be a regular cardinal.
  \begin{enumerate-(1)}
    \item The challenger has a winning strategy in $G_\gamma^\theta(\kappa)$ iff he has a winning strategy in $\overline{G_\gamma^\theta}(\kappa)$.
    \item The judge has a winning strategy in $G_\gamma^\theta(\kappa)$ iff she has a winning strategy in $\overline{G_\gamma^\theta}(\kappa)$.
  \end{enumerate-(1)}
\end{lemma}
\begin{proof}
  If the challenger has a winning strategy in $G_\gamma^\theta(\kappa)$, then he has one in $\overline{G_\gamma^\theta}(\kappa)$, as the latter game only gives less choice for the judge. Assume the challenger has a winning strategy $\bar S$ in $\overline{G_\gamma^\theta}(\kappa)$. Let $S$ be the strategy for $G_\gamma^\theta(\kappa)$ where the challenger pretends that the judge had played $F_i\cap M_i$ rather than $F_i$, at every stage $i$ of a play of $G_\gamma^\theta(\kappa)$, and the challenger responds according to that, following the strategy $\bar S$. This yields a run of the game $\overline{G_\gamma^\theta}(\kappa)$ where the challenger follows his winning strategy, hence the judge loses this play, i.e.\ $F_\gamma\cap M_\gamma$ is not $M_\gamma$-normal. But then the same is the case for $F_\gamma$, i.e.\ $S$ is a winning strategy for the challenger in the game $G_\gamma^\theta(\kappa)$.
  
 If the judge has a winning strategy in $\overline{G_\gamma^\theta}(\kappa)$, then this is also a winning strategy in $G_\gamma^\theta(\kappa)$. If she has a winning strategy $S$ in $G_\gamma^\theta(\kappa)$, let $\bar S$ be the modification where rather than playing $F_i$, she plays $F_i\cap M_i$, at each stage $i<\gamma$. 
 Since $S$ is a winning strategy, $F_\gamma$ is $M_\gamma$-normal, whenever it is the outcome of a play of $G_\gamma^\theta(\kappa)$. But then also $F_\gamma\cap M_\gamma$ is $M_\gamma$-normal. Hence $\bar S$ is also a winning strategy for $G_\gamma^\theta(\kappa)$. But every play of $G_\gamma^\theta(\kappa)$ following $\bar S$ is also a run of the game $\overline{G_\gamma^\theta}(\kappa)$, i.e.\ $\bar S$ is a winning strategy for $\overline{G_\gamma^\theta}(\kappa)$.
\end{proof}

\begin{lemma}\label{cardinalsequivalence}
  Let $\gamma\le\kappa^+$, let $\kappa=\kappa^{<\kappa}$ be an uncountable cardinal, and let $\theta_0$ and $\theta_1$ both be regular cardinals greater than $\kappa$.
  \begin{enumerate-(1)}
    \item The challenger has a winning strategy in $G_\gamma^{\theta_0}(\kappa)$ iff he has a winning strategy in $G_\gamma^{\theta_1}(\kappa)$.
    \item The judge has a winning strategy in $G_\gamma^{\theta_0}(\kappa)$ iff she has a winning strategy in $G_\gamma^{\theta_1}(\kappa)$.
  \end{enumerate-(1)}
\end{lemma}
\begin{proof}
Let $\gamma$ be an ordinal, and assume that $\theta_0$ and $\theta_1$ are both regular cardinals greater than $\kappa$. For (1), assume that the challenger has a winning strategy $\sigma_0$ in $G_\gamma^{\theta_0}(\kappa)$. We show that he then has a winning strategy $\sigma_1$ in $G_\gamma^{\theta_1}(\kappa)$. $\sigma_1$ is obtained as follows. Whenever the challenger would play $M_\alpha$ in a run of the game $G_\gamma^{\theta_0}(\kappa)$, then he plays some $M_\alpha^*$ which is a valid move in the game $G_\gamma^{\theta_1}(\kappa)$ and such that $M_\alpha^*\supseteq\mathcal P(\kappa)\cap M_\alpha$. Every possible response of the judge in $G_\gamma^{\theta_1}(\kappa)$ is also a possible response in $G_\gamma^{\theta_0}(\kappa)$, where the challenger played $M_\alpha$. So the challenger can continue to pretend playing both these games simultaneously. As he is following a winning strategy in the game $G_\gamma^{\theta_0}(\kappa)$, $F_\gamma$ is not $M_\gamma$-normal. But then $F_\gamma$ is not $\bigcup_{\alpha<\gamma}M_\alpha^*$-normal either. This shows that $\sigma_1$ is a winning strategy for the challenger in the game $G_\gamma^{\theta_1}(\kappa)$.

  For (2), assume that the judge has a winning strategy $\sigma_0$ in $G_\gamma^{\theta_0}(\kappa)$. We show that she then has a winning strategy $\sigma_1$ in $G_\gamma^{\theta_1}(\kappa)$. $\sigma_1$ is obtained by simply pretending that, if the challenger plays $M_\alpha$ at any stage $\alpha$ of the game $G_\gamma^{\theta_1}(\kappa)$, he in fact played some $M_\alpha^*$ in the game $G_\gamma^{\theta_0}(\kappa)$ with the property that $M_\alpha^*\supseteq M_\alpha\cap\mathcal P(\kappa)$, and respond according to that. Since $\sigma_0$ is a winning strategy for the judge in the game $G_\gamma^{\theta_0}(\kappa)$, $F_\gamma$ is $\bigcup_{\alpha<\gamma}M_{\alpha}^*$-normal. But then $F_\gamma$ will also be $M_\gamma$-normal. This shows that $\sigma_1$ is a winning strategy for the judge in $G_\gamma^{\theta_1}(\kappa)$.
\end{proof}

In the light of Lemma \ref{cardinalsequivalence}, we can make the following definition.

\begin{definition}\label{def:filterproperties}
Suppose $\kappa=\kappa^{<\kappa}$ is an uncountable cardinal, $\theta>\kappa$ is a regular cardinal, and $\gamma\le\kappa^+$. 
\begin{enumerate-(a)} 
\item 
$\kappa$ has the \emph{$\gamma$-filter property} if the challenger does not have a winning strategy in $G_\gamma^\theta(\kappa)$.
\item
$\kappa$ has the \emph{strategic $\gamma$-filter property} if the judge has a winning strategy in $G_\gamma^\theta(\kappa)$. 
\end{enumerate-(a)} 
\end{definition} 

The $1$-filter property follows from weak compactness by its embedding characterization, and implies the filter property, hence it is equivalent to weak compactness. Note that if $\gamma_0<\gamma_1$, then the $\gamma_1$-filter property implies the $\gamma_0$-filter property. The following observation shows that assuming $2^\kappa=\kappa^+$, the $\kappa^+$-filter property is equivalent to $\kappa$ being a measurable cardinal.

\begin{observation} 
  The following are equivalent for any uncountable cardinal $\kappa=\kappa^{<\kappa}$ satisfying $2^\kappa=\kappa^+$.
\begin{enumerate-(a)}
  \item $\kappa$ satisfies the $\kappa^+$-filter property.
  \item $\kappa$ satisfies the strategic $\kappa^+$-filter property.
  \item $\kappa$ is measurable. \footnote{One could extend our definitions in a natural way so to give rise to the concept of $\kappa$ having the $\gamma$-filter property also for ordinals $\gamma>\kappa^+$, essentially dropping the requirement that the models played by the challenger have size $\kappa$. This would however make our definitions less elegant, and was omitted for we will mostly be interested in the case when $\gamma\le\kappa$ in what follows. However right now, these extended definitions would yield the more elegant observation that $\kappa$ being measurable is equivalent to it having the (strategic) $2^\kappa$-filter property.}
\end{enumerate-(a)}
\end{observation} 
\begin{proof} 
  For the implication from (a) to (c), suppose that $\kappa$ has the $\kappa^+$-filter property, and that $\langle a_\alpha\mid\alpha< \kappa^+\rangle$ is an enumeration of $\mathcal P(\kappa)$. Let $\theta>\kappa$ be an arbitrary regular cardinal. We consider a run of the game $G_{\kappa^+}^\theta(\kappa)$ such that in each step $\alpha$, the challenger plays a valid $M_\alpha\supseteq\{a_\beta\mid\beta\le\alpha\}$, however the judge wins. Then, $F_\gamma$ is a normal ultrafilter on $\mathcal P(\kappa)$. 

To see that (c) implies (b), suppose that $\kappa$ is measurable and let $F$ be a ${<}\kappa$-complete ultrafilter on $\mathcal P(\kappa)$. Then, for any regular $\theta>\kappa$, the judge wins any run of $G_{\kappa^+}^\theta(\kappa)$ by playing $F$ in each of her moves.

Finally, the implication from (b) to (a) is immediate.
\end{proof} 

We will show that the $\alpha$-filter properties for infinite cardinals $\alpha$ with $\omega\le\alpha\le\kappa$ give rise to a proper hierarchy of large cardinal notions, that are closely related to the following Ramsey-like cardinals, that were introduced by Victoria Gitman in \cite{MR2830415}.

\section{Victoria Gitman's Ramsey-like cardinals}\label{gitman}

\begin{definition}\quad
\begin{enumerate-(a)}
  \item \cite[Definition 1.2]{MR2830415} A cardinal $\kappa$ is \emph{weakly Ramsey} if every $A\subseteq\kappa$ is contained, as an element, in a weak $\kappa$-model $M$ for which there exists a $\kappa$-powerset preserving elementary embedding $j\colon M\to N$.
  \item \cite[Definition 1.4]{MR2830415} A cardinal $\kappa$ is \emph{strongly Ramsey} if every $A\subseteq\kappa$ is contained, as an element, in a $\kappa$-model $M$ for which there exists a $\kappa$-powerset preserving elementary embedding $j\colon M\to N$.
  \item \cite[Definition 1.5]{MR2830415} A cardinal $\kappa$ is \emph{super Ramsey} if every $A\subseteq\kappa$ is contained, as an element, in a $\kappa$-model $M\prec H(\kappa^+)$ for which there exists a $\kappa$-powerset preserving elementary embedding $j\colon M\to N$.
\end{enumerate-(a)}
\end{definition}

The following proposition is an immediate consequence of \cite[Theorem 3.7]{MR2830415}, where Gitman shows that weakly Ramsey cardinals are limits of \emph{completely ineffable} cardinals (see \cite[Definition 3.4]{MR2830415}). It yields in particular that weak Ramseyness is strictly stronger than weakly compactness.

\begin{proposition}\cite{MR2830415}
  Weakly Ramsey cardinals are weakly compact limits of ineffable cardinals.
\end{proposition}

The following theorem from \cite{MR2830415}, which is already implicit in \cite{MR534574}, 
shows that strongly Ramsey cardinals are Ramsey cardinals, which in turn are weakly Ramsey. In fact, as is shown in \cite[Theorems 3.9 and 3.11]{MR2830415}, strongly Ramsey cardinals are Ramsey limits of Ramsey cardinals, and Ramsey cardinals are weakly Ramsey limits of weakly Ramsey cardinals.

\begin{theorem}\cite[Theorem 1.3]{MR2830415}\label{ramseyembeddings}
  A cardinal $\kappa$ is Ramsey if and only if every $A\subseteq\kappa$ is contained, as an element, in a weak $\kappa$-model $M$ for which there exists a $\kappa$-powerset preserving elementary embedding $j\colon M\to N$ with the additional property that whenever $\langle A_n\mid n\in\omega\rangle$ is a sequence of subsets of $\kappa$ (that is not necessarily an element of $M$) such that for each $n\in\omega$, $A_n\in M$ and $\kappa\in j(A_n)$, then $\bigcap_{n\in\omega}A_n\ne\emptyset$.
\end{theorem}

\begin{proposition}\cite[Theorem 3.14]{MR2830415}
  Super Ramsey cardinals are strongly Ramsey limits of strong\-ly Ramsey cardinals.
\end{proposition}

A notion that is closely related to the above, that however was not introduced in \cite{MR2830415}, is the strengthening of weak Ramseyness where we additionally require the witnessing structures $M$ to be elementary substructures of $H(\kappa^+)$, like Gitman does when strengthening strongly Ramsey to super Ramsey cardinals. We make the following definition.

\begin{definition}
  A cardinal $\kappa$ is \emph{super weakly Ramsey} if every $A\subseteq\kappa$ is contained, as an element, in a weak $\kappa$-model $M\prec H(\kappa^+)$ for which there exists a $\kappa$-powerset preserving elementary embedding $j\colon M\to N$.
\end{definition}

\begin{proposition}
  Super weakly Ramsey cardinals are weakly Ramsey limits of weakly Ramsey cardinals.
\end{proposition}
\begin{proof}
  Suppose that $\kappa$ is super weakly Ramsey, and pick a weak $\kappa$-model $M\prec H(\kappa^+)$ and a $\kappa$-powerset preserving elementary embedding $j\colon M\to N$. It suffices to show that $\kappa$ is weakly Ramsey in $N$. But as we can assume that the models witnessing instances of weak Ramseyness of $\kappa$ are all elements of $H(\kappa^+)$, $M$ thinks that $\kappa$ is weakly Ramsey by elementarity, and hence $N$ thinks that $\kappa$ is weakly Ramsey for $j$ is $\kappa$-powerset preserving.
\end{proof}

As is observed in \cite{MR2830415}, since ineffable cardinals are $\mathbf\Pi^1_2$-indescribable and being Ramsey is a $\mathbf\Pi^1_2$-statement, ineffable Ramsey cardinals are limits of Ramsey cardinals. Thus in particular not every Ramsey cardinal is ineffable. However the following holds true.

\begin{proposition}\label{swr_ineffable}
  Super weakly Ramsey cardinals are ineffable.
\end{proposition}
\begin{proof}
  Assume that $\kappa$ is super weakly Ramsey. Let $\vec A=\langle A_\alpha\mid\alpha<\kappa\rangle$ be a $\kappa$-list, and let $j\colon M\to N$ be $\kappa$-powerset preserving with $M\prec H(\kappa^+)$ and $\vec A\in M$. Let $A=j(\vec A)(\kappa)$. Then $A\in M$, since $j$ is $\kappa$-powerset preserving. Let $S=\{\alpha<\kappa\mid A\cap\alpha=A_\alpha\}\in M$. Let $C$ be a club subset of $\kappa$ in $M$. Then $\kappa\in j(S)\cap j(C)$, and thus $C\cap S\ne\emptyset$ by elementarity of $j$, showing that $S$ is a stationary subset of $\kappa$ in $M$. But since $M\prec H(\kappa^+)$, $S$ is indeed stationary, thus showing that $\kappa$ is ineffable, as desired.
\end{proof}

\section{A new hierarchy of Ramsey-like cardinals}\label{ramseylikehierarchy}

We want to introduce the following hierarchy of Ramsey-like cardinals.

\begin{definition}\label{alphaRamsey}	
  Let $\alpha\le\kappa$ be regular cardinals. $\kappa$ is \emph{$\alpha$-Ramsey} if for arbitrarily large regular cardinals $\theta$, every $A\subseteq\kappa$ is contained, as an element, in some weak $\kappa$-model $M\prec H(\theta)$ which is closed under ${<}\alpha$-sequences, and for which there exists a $\kappa$-powerset preserving elementary embedding $j\colon M\to N$.
\end{definition}

Note that, in the case when $\alpha=\kappa$, a weak $\kappa$-model closed under ${<}\kappa$-sequences is exactly a $\kappa$-model. It would have been more in the spirit of \cite{MR2830415}, and in stronger analogy to Gitman's super Ramsey cardinals, to only require the above for $\theta=\kappa^+$. However we will argue that asking for the existence of arbitrary large $\theta>\kappa$ as above results in a more natural (and strictly stronger) notion.

\begin{proposition}\label{kappasuperramsey}
  If $\kappa$ is $\kappa$-Ramsey, then $\kappa$ is a super Ramsey limit of super Ramsey cardinals.
\end{proposition}
\begin{proof}
  Assume that $\kappa$ is $\kappa$-Ramsey, as witnessed by some large regular cardinal $\theta$ and $j\colon M\to N$ with $M\prec H(\theta)$. Since $\kappa^+\in M$, it follows that the restriction of $j$ to $H(\kappa^+)^M$ witnesses that $\kappa$ is super Ramsey in $V$. It thus suffices to show that $\kappa$ is super Ramsey in $N$.
  
By elementarity, $M$ thinks that $\kappa$ is super Ramsey. However, as the target structures of embeddings witnessing super Ramseyness can be assumed to be elements of $H(\kappa^+)$, this is a statement which is absolute between weak $\kappa$-models with the same subsets of $\kappa$ (and thus the same $H(\kappa^+)$) that contain $\kappa^+$ as an element, hence $\kappa$ is super Ramsey in $N$, using that $j$ is $\kappa$-powerset preserving.
\end{proof}

Unsurprisingly, the same proof yields the analogous result for $\omega$-Ramsey and super weakly Ramsey cardinals. Note that together with Proposition \ref{swr_ineffable} and the remarks preceding it, the following proposition shows in particular that Ramsey cardinals are not provably $\omega$-Ramsey.

\begin{proposition}
  If $\kappa$ is $\omega$-Ramsey, then $\kappa$ is a super weakly Ramsey limit of super weakly Ramsey cardinals. \hfill{$\Box$}
\end{proposition}

%
%

\begin{proposition}\label{w1RR}
  If $\kappa$ is $\omega_1$-Ramsey, then $\kappa$ is a Ramsey limit of Ramsey cardinals.
\end{proposition}
\begin{proof}
  Suppose that $\kappa$ is $\omega_1$-Ramsey. Then $\kappa$ is Ramsey, as the witnessing models for $\omega_1$-Ramseyness are closed under countable sequences, and thus also witness the respective instances of Ramseyness.
  Pick a sufficiently large regular cardinal $\theta$, a weak $\kappa$-model $M\prec H(\theta)$ and $j\colon M\to N$ witnessing the $\omega_1$-Ramseyness of $\kappa$ for $A=\emptyset$. Note that Ramseyness of $\kappa$ is, considering only transitive weak $\kappa$-models, which suffices, a statement about $H(\kappa^+)$ and thus $\kappa$ is Ramsey in $M$. Since $j$ is $\kappa$-powerset preserving, $\kappa$ is also Ramsey in $N$, for the same reason. But this implies, by elementarity, that $\kappa$ is a limit of Ramsey cardinals, both in $M$ and in $V$.
\end{proof}

In \cite{MR1077260}, Feng introduces a hierarchy of Ramsey cardinals that he denotes as $\mathbf\Pi_\alpha$-Ramsey, for $\alpha\in\On$ (these have also been called $\alpha$-Ramsey cardinals in \cite{MR2817562}). This hierarchy is topped by the notion of what he calls a \emph{completely Ramsey} cardinal. This hierarchy is not so much of interest to us here, as already $\omega_1$-Ramsey cardinals are completely Ramsey limits of completely Ramsey cardinals. This follows from elementarity together with the proof of \cite[Theorem 3.13]{MR2830415}, observing that rather than using a $\kappa$-model $M$, using a weak $\kappa$-model $M$ that is closed under $\omega$-sequences suffices to run the argument. Note that by \cite[Theorem 4.2]{MR1077260}, completely Ramsey cardinals are $\mathbf\Pi^2_0$-indescribable, thus in particular this implies that $\omega_1$-Ramsey cardinals are $\mathbf\Pi^2_0$-indescribable as well.

The next lemma will show that $\alpha$-Ramseyness is a very robust notion, for any regular cardinal $\alpha\le\kappa$. This will be given additional support by a filter game characterization of $\alpha$-Ramseyness for uncountable cardinals $\alpha$ in Theorem \ref{ramseyfilterequivalence} and Corollary \ref{wramseyfilterequivalence} below.

\begin{theorem}\label{modestequivalences}
Let $\alpha\le\kappa$ be regular cardinals. The following properties are equivalent. 
\begin{enumerate-(a)} 
\item $\kappa$ is $\alpha$-Ramsey.
\item For arbitrarily large regular cardinals $\theta$, every $A\subseteq\kappa$ is contained, as an element, in a weak $\kappa$-model $M\prec H(\theta)$ that is closed under ${<}\alpha$-sequences, and for which there exists a good $M$-normal filter on $\kappa$.
\item Like (a) or (b), but $A$ can be any element of $H(\theta)$.
\item Like (a) or (b), but only for $A=\emptyset$.
\end{enumerate-(a)}
If $\alpha>\omega$, the following property is also equivalent to the above.
\begin{enumerate-(a)}
  \item[(e)] Like (c), but only for a single regular $\theta\ge(2^\kappa)^+$. 
\end{enumerate-(a)}
\end{theorem} 
\begin{proof}
  The equivalence of (a) and (b), as well as the equivalences of the versions of (c), (d) and (e) that refer to (a) to their respective counterparts that refer to (b) are immediate consequences of Lemma \ref{wapp_equivalence} together with \cite[Proposition 2.3]{MR2830415}. Clearly, (c) implies (a), and (a) implies each of (d) and (e). The proof of the implication from (e) to (a) 
  will be postponed to Lemma \ref{finalimplication} below. We will now show that (d) implies (c).

Therefore, suppose that (d) holds, and let us suppose for a contradiction that there is some regular $\theta>\kappa$ and some $A\in H(\theta)$, such that no $M$, $N$ and $j$ witnessing (c) for $\theta$ and $A$ exist. Choose a regular cardinal $\theta'$, large enough so that this can be seen in $H(\theta')$, i.e.\ \[H(\theta')\!\models\exists\theta\!>\!\kappa\textrm{ regular}\,\exists A\!\in\!H(\theta)\,\forall M\,\forall j\,\forall N\, [(M\!\prec\!H(\theta)\textrm{ is a weak }\kappa\textrm{-model}\]\[\textrm{with }M^{<\alpha}\subseteq M\,\land\,j\colon M\to N\textrm{ is }\kappa\textrm{-powerset preserving})\,\to\,A\not\in M],\]
such that the above statement is absolute between $H(\theta')$ and $V$ for the least witness $\theta$ and any $A$ in $H(\theta)$, and such that (d) holds for $\theta'$. The absoluteness statement can easily be achieved, noting that it suffices to consider transitive models $N$ of size $\kappa$. Making use of Property (d), there is a weak $\kappa$-model $M_1\prec H(\theta')$ and a $\kappa$-powerset preserving embedding $j\colon M_1\to N_1$. By elementarity, $M_1$ models the above statement about $H(\theta')$, thus in particular we can find the least $\theta$ and some $A\in H(\theta)$ witnessing the above statement in $M_1$. Since $\theta\in M_1$, $M_1\cap H(\theta)\prec H(\theta)$, $A\in M_1\cap H(\theta)$ and $j\upharpoonright(H(\theta)^{M_1})\colon H(\theta)^{M_1}\to H(j(\theta))^{N_1}$ is $\kappa$-powerset preserving, contradicting our assumption about $\theta$ and $A$.
\end{proof}

\begin{theorem}\label{ramseyfilterequivalence} 
 Let $\alpha\le\kappa$ be regular and uncountable cardinals. Then $\kappa$ is $\alpha$-Ramsey if and only if $\kappa=\kappa^{<\kappa}$ has the $\alpha$-filter property.
\end{theorem} 
\begin{proof} 
Assume first that $\kappa$ has the $\alpha$-filter property. Pick some large regular cardinal $\theta$. Let $A\subseteq\kappa$ and pick any strategy for the challenger in the game $G_\alpha^\theta(\kappa)$, such that $A$ is an element of the first model played.
Since the challenger has no winning strategy in the game $G_\alpha^\theta(\kappa)$ by our assumption, there is a run of this game where the challenger follows the above strategy, however the judge wins. Let $\langle M_\gamma\mid\gamma<\alpha\rangle$ and $\langle F_\gamma\mid\gamma<\alpha\rangle$ be the moves made during such a run, let $F_\alpha$ and $M_\alpha$ be their unions.
By the regularity of $\alpha$, $M_\alpha$ is a weak $\kappa$-model that is closed under ${<}\alpha$-sequences. Since the judge wins, $F_\alpha$ is an $M_\alpha$-normal filter. Since $\alpha>\omega$, $F_\alpha$ induces a well-founded ultrapower of $M_\alpha$.  
It remains to show that $F_\alpha$ is weakly amenable for $M_\alpha$. Therefore, assume that $X\subseteq\mathcal P(\kappa)$ is of size at most $\kappa$ in $M_\alpha$. By the definition of $M_\alpha$, this is the case already in $M_\gamma$, for some $\gamma<\alpha$. But since $F_\gamma\cap M_\gamma\in M_{\gamma+1}$, $F_\alpha\cap X=F_\gamma\cap X\in M_{\gamma+1}\subseteq M_\alpha$, showing that $F_\alpha$ is weakly amenable and hence good, i.e.\ $\kappa$ is $\alpha$-Ramsey.

Now assume that $\kappa$ is $\alpha$-Ramsey and let $\theta=(2^\kappa)^+$.
Towards a contradiction, suppose that the challenger has a winning strategy $\sigma$ in $\overline{G_\alpha^\theta}(\kappa)$. Then $\sigma\in H(\theta)$. Since $\kappa$ is $\alpha$-Ramsey, there is a weak $\kappa$-model $M\prec H(\theta)$ that is closed under ${<}\alpha$-sequences, with $\sigma\in M$, and a good $M$-normal filter $U$ on $\kappa$. 
We define a partial strategy $\tau$ for the judge in $\overline{G_\alpha^\theta}(\kappa)$ as follows. If the challenger played $M_\gamma\prec H(\theta)$, with $M_\gamma\in M$, in his last move, then the judge answers by playing $F_\gamma=U\cap M_\gamma$. Note that $F_\gamma\in M$, since $U$ is weakly $M$-amenable.
Since $\sigma\in M$, the above together with closure of $M$ under ${<}\alpha$-sequences implies that the run of $\sigma$ against $\tau$ has length $\alpha$, since all its initial segments of length less than $\alpha$ are elements of $M$. Note that $F_\alpha$ is an $M_\alpha$-normal filter, that gives rise to a well-founded ultrapower of $M_\alpha$. Thus using her (partial) strategy $\tau$, the judge wins against $\sigma$, contradicting the assumption that $\sigma$ is a winning strategy for the challenger in $\overline{G_\alpha^\theta}(\kappa)$. By Lemma \ref{barequivalence} and Lemma \ref{cardinalsequivalence}, this implies that $\kappa$ has the $\alpha$-filter property.
\end{proof} 

To obtain a version of Theorem \ref{ramseyfilterequivalence} for $\omega$-Ramsey cardinals, we make the following, somewhat ad hoc definitions.

\begin{definition}
  Suppose $\kappa=\kappa^{<\kappa}$ is an uncountable cardinal, $\theta>\kappa$ is a regular cardinal, and $\gamma\le\kappa^+$.  We define the \emph{well-founded filter games} $wfG_\gamma^\theta(\kappa)$ just like the filter games $G_\gamma^\theta(\kappa)$ in Definition \ref{filtergamesdefinition}, however for the judge to win, we additionally require that the ultrapower of $M_\gamma$ by $F_\gamma$ be well-founded. \footnote{Note that in case $\gamma$ has uncountable cofinality, $M_\gamma$ will always be closed under countable sequences and thus this extra condition becomes vacuous.} We say that $\kappa$ has the \emph{well-founded $(\gamma,\theta)$-filter property} if the challenger does not have a winning strategy in $wfG_\gamma^\theta(\kappa)$. We say that $\kappa$ has the \emph{well-founded $\gamma$-filter property} iff it has the well-founded $(\gamma,\theta)$-filter property for every regular $\theta>\kappa$. \footnote{Very recently, Victoria Gitman has shown that the well-founded $\omega$-filter property is strictly stronger than the $\omega$-filter property -- see Lemma \ref{victoriaomegafilterproperty} below.}
\end{definition}

The proof of Theorem \ref{ramseyfilterequivalence} also shows the following, where in the forward direction, well-foundedness of the ultrapower of $M_\omega$ by $F_\omega$ now follows from the well-founded $\omega$-filter property rather than the (now missing) closure properties of $M_\omega$.

\begin{corollary}\label{wramseyfilterequivalence}
  $\kappa$ is $\omega$-Ramsey iff $\kappa=\kappa^{<\kappa}$ has the well-founded $\omega$-filter property. \hfill{$\Box$}
\end{corollary}

We can now use the above to fill in the missing part of the proof of Theorem \ref{modestequivalences}.

\begin{lemma}\label{finalimplication}
  For regular cardinals $\alpha\le\kappa$, Property (e) implies Property (a) in the statement of Theorem \ref{modestequivalences}.
\end{lemma}
\begin{proof}
  Note that when showing that $\kappa$ being $\alpha$-Ramsey implies the $\alpha$-filter property in the proof of Theorem \ref{ramseyfilterequivalence}, we only used the case when $\theta=(2^\kappa)^+$, and in fact it would have worked for any regular $\theta\ge(2^\kappa)^+$ in the very same way. Thus our assumption implies the $\alpha$-filter property. But then again by Lemma \ref{ramseyfilterequivalence}, $\kappa$ is $\alpha$-Ramsey, as desired.
\end{proof}

We think that the above results in particular show $\kappa$-Ramseyness to be a more natural large cardinal notion than the closely related concept of super Ramseyness defined by Gitman - super Ramseyness corresponds to Property (e) for $\theta=\kappa^+$ in Theorem \ref{modestequivalences} above, while what may seem to be a hierarchy for different $\theta\ge(2^\kappa)^+$ in Property (e) of Theorem \ref{modestequivalences} actually collapses to the single notion of $\kappa$-Ramseyness.

\medskip

While $\alpha$-Ramseyness for singular cardinals $\alpha$ is not a very useful property, as it implies $\alpha^+$-Ramseyness (since weak $\kappa$-models closed under ${<}\alpha$-sequences are also closed under ${<}\alpha^+$-sequences), the $\alpha$-filter property makes perfect sense also when $\alpha$ is singular. We may thus define, for singular cardinals $\alpha$, that $\kappa$ is $\alpha$-Ramsey if it has the $\alpha$-filter property. For the cases when $\alpha$ has cofinality $\omega$, we may rather want to consider the well-founded $\alpha$-filter property instead.

\medskip

We now want to show that the $\alpha$-Ramsey cardinals (including those we just defined for singular cardinals $\alpha$) form a strict hierarchy for cardinals $\omega\le\alpha\le\kappa$, and moreover that $\kappa$-Ramsey cardinals are strictly weaker than measurable cardinals.

\begin{theorem}\label{filterprophierarchy}
  If $\omega\le\alpha_0<\alpha_1\le\kappa$, both $\alpha_0$ and $\alpha_1$ are cardinals, and $\kappa$ is $\alpha_1$-Ramsey, then there is a proper class of $\alpha_0$-Ramsey cardinals in $\VV_\kappa$. If $\alpha_0$ is regular, then $\kappa$ is a limit of $\alpha_0$-Ramsey cardinals.
\end{theorem}
\begin{proof}
  Pick a regular cardinal $\theta>\kappa$. We may assume that $\alpha_1$ is regular, for we may replace it with a regular $\bar{\alpha_1}$ that lies strictly between $\alpha_0$ and $\alpha_1$ otherwise. Using that $\kappa$ is $\alpha_1$-Ramsey, pick an ultrapower embedding $j\colon M\to N$ where $M\prec H(\theta)$ is a weak $\kappa$-model that is closed under ${<}\alpha_1$-sequences, and $j$ is $\kappa$-powerset preserving. We may also assume that $N$ is transitive, since we can replace it by its transitive collapse in case it is not. Using that $j$ is an ultrapower embedding, it follows by standard arguments that $N$ is closed under ${<}\alpha_1$-sequences as well. Moreover, $j$ induces a weakly amenable $M$-normal filter $F$, by Lemma \ref{wapp_equivalence}, (2). By $\kappa$-powerset preservation of $j$, $F$ is also weakly amenable for $N$ and $N$-normal. Let $\nu>\kappa$ be a regular cardinal of $N$. We show that $\kappa$ has the well-founded $(\alpha_0,\nu)$-filter property in $N$.

Suppose for a contradiction that the challenger has a winning strategy for $wfG_{\alpha_0}^\nu(\kappa)$ in $N$, and let him play according to this strategy. Whenever he plays a $\kappa$-model $X\prec H(\nu)$, let the judge answer by playing $F\cap X\in N$. By closure of $N$ under ${<}\alpha_1$-sequences, this yields a run of the game $wfG_{\alpha_0}^\nu(\kappa)$ that is an element of $N$. Moreover, the judge wins this run: If $Y$ denotes the union of the models played by the challenger, potential ill-foundedness of the ultrapower of $Y$ by $F\cap Y$ would be witnessed by a sequence $\langle f_i\mid i<\omega\rangle$ of functions $f_i\colon\kappa\to Y$ in $Y$, for which $F_i=\{\alpha<\kappa\mid f_{i+1}(\alpha)\in f_i(\alpha)\}\in F$ for every $i<\omega$. Now by transitivity of $N$ and since $N$ is closed under $\omega$-sequences, $\langle f_i\mid i<\omega\rangle\in N$. But then since $F$ is $N$-normal, $\bigcap_{i<\omega}F_i\in F$, yielding a decreasing $\omega$-sequence of ordinals in $N$, a contradiction. This means that the ultrapower of $Y$ by $F\cap Y$ is well-founded, i.e.\ the judge wins the above run of the game $wfG_{\alpha_0}^\nu(\kappa)$. However this contradicts that the challenger followed his winning strategy.

The first statement of the theorem now follows by elementarity together with Theorem \ref{ramseyfilterequivalence}, and its second statement follows immediately from the regularity of $\alpha_0$ together with the relevant definitions.
\end{proof}

\begin{proposition} 
If $\kappa$ is measurable, then it is a limit of regular cardinals $\alpha<\kappa$ which are $\alpha$-Ramsey.
\end{proposition} 
\begin{proof} 
Assume that $\kappa$ is measurable, as witnessed by $j\colon V\to M$. Using that $M$ is closed under $\kappa$-sequences, the proof now proceeds like the proof of Theorem \ref{filterprophierarchy}.
\end{proof} 

\section{Filter sequences}\label{section:filtersequences}

In this section, we show that the filter properties, which are based on (the non-existence of) winning strategies for certain games, are closely related to certain principles that are solely based on the existence of certain sequences of models and filters.

\begin{definition} 
Let $\alpha$ be an ordinal and let $\kappa$ be a cardinal. Suppose that $\vec{M}=\langle M_i\mid i<\alpha\rangle$ is a $\subseteq$-increasing $\in$-chain of $\kappa$-models, and let $M=\bigcup_{i<\alpha}M_i$. An $M$-normal filter $F$ on $\kappa$ is \emph{amenable for $\vec{M}$} if $F\cap M_i\in M_{i+1}$ for all $i<\alpha$. If such an $\alpha$-sequence $\vec M$ and such an $M$-normal filter $F$ exist, we say that \emph{$\kappa$ has an $\alpha$-filter sequence}. If additionally the ultrapower of $M$ by $F$ is well-founded, we say that \emph{$\kappa$ has a well-founded $\alpha$-filter sequence}. \footnote{As before this additional assumption becomes vacuous if $\alpha$ has uncountable cofinality.} 
\end{definition} 

Observe that if $\alpha$ is a limit ordinal and $F$ is a filter on $\kappa$ that is amenable for an $\in$-chain $\vec M=\langle M_i\mid i<\alpha\rangle$ of weak $\kappa$-models, then letting $M=\bigcup_{i<\alpha}M_i$, $F$ is weakly amenable for $M$
.

The following is immediate by Theorem \ref{ramseyfilterequivalence} and Corollary \ref{wramseyfilterequivalence}.

\begin{observation}
  Assume that $\alpha\le\kappa$ are both cardinals, and $\kappa$ is $\alpha$-Ramsey. Then $\kappa$ has a well-founded $\alpha$-filter sequence.
\end{observation}

The next proposition shows that consistency-wise, the existence of (well-founded) $\alpha$-filter sequences forms a proper hierarchy for infinite cardinals $\alpha\le\kappa$, that interleaves with the hierarchy of $\alpha$-Ramsey cardinals. Its proof is similar to the proof of Theorem \ref{filterprophierarchy}.

\begin{proposition}
Suppose that $\omega\le\alpha<\beta\le\kappa$ are cardinals, and that $\kappa$ has a $\beta$-filter sequence. Then there is a proper class of $\alpha$-Ramsey cardinals in $\VV_\kappa$. If $\alpha$ is regular, then $\kappa$ is a limit of $\alpha$-Ramsey cardinals.
\end{proposition} 
\begin{proof}
  We may assume that $\beta$ is regular, for we may replace it with a regular $\bar\beta$ that lies strictly between $\alpha$ and $\beta$ otherwise. Suppose that $\kappa$ has a $\beta$-filter sequence, as witnessed by $\vec{M}=\langle M_i\mid i<\beta\rangle$, $M=\bigcup_{i<\beta}M_i$, and by the $M$-normal filter $F$. Let $N$ be the well-founded ultrapower of $M$ by $F$, using that $M$ is closed under ${<}\beta$-sequences, and note that since $\mathcal P(\kappa)^M=\mathcal P(\kappa)^N$, $F$ is weakly amenable for $N$ and $N$-normal. Note that $N$ is also closed under ${<}\beta$-sequences. Let $\nu>\kappa$ be a regular cardinal in $N$. Then $\kappa$ has the $(\alpha,\nu)$-filter property in $N$, since the judge can win any relevant (well-founded) filter game in $N$ by playing appropriate $\kappa$-sized pieces of $F$, just like in the proof of Theorem \ref{filterprophierarchy}. 
  
  As in that proof, the first statement of the proposition now follows by elementarity together with Theorem \ref{ramseyfilterequivalence}, and its second statement follows immediately from the regularity of $\alpha_0$ together with the relevant definitions.
\end{proof}

\begin{observation} 
The existence of a $\kappa$-filter sequence does not imply that $\kappa$ is weakly compact. 
\end{observation} 
\begin{proof}
  Start in a model with a $\kappa$-filter sequence in which $\kappa$ is also weakly compact. Perform some forcing of size less than $\kappa$. This preserves both these properties of $\kappa$. Now by \cite[Main Theorem]{MR1607499}, there is a ${<}\kappa$-closed forcing that destroys the weak compactness of $\kappa$ over this model. Clearly this forcing preserves the existence of the $\kappa$-filter sequence that we started with.
\end{proof} 

However for regular cardinals $\alpha$, we can actually characterize $\alpha$-Ramsey cardinals by the existence of certain filter sequences. Note that the following proposition is highly analogous to Theorem \ref{modestequivalences}, and that some more equivalent characterizations of $\alpha$-Ramseyness could be extracted from the proof of that theorem, similar to the ones below.

\begin{proposition} 
The following are equivalent, for regular cardinals $\alpha\le\kappa$.
\begin{enumerate-(a)} 
\item $\kappa$ is $\alpha$-Ramsey.
\item For every regular $\theta>\kappa$, $\kappa$ has an $\alpha$-filter sequence, as witnessed by $\vec M=\langle M_i\mid i<\alpha\rangle$ and $F$, where each $M_i\prec H(\theta)$.
\end{enumerate-(a)} 
If $\alpha>\omega$, the following property is also equivalent to the above.
\begin{enumerate-(a)}
\item[$(c)$] For some regular $\theta>2^\kappa$ and every $A\subseteq\kappa$, $\kappa$ has an $\alpha$-filter sequence, as witnessed by $\vec M=\langle M_i\mid i<\alpha\rangle$ and $F$, where $A\in M_0$ and each $M_i\prec H(\theta)$.
\end{enumerate-(a)}
\end{proposition} 
\begin{proof}
  If $\kappa$ is $\alpha$-Ramsey, then both (b) and (c) are immediate by the proof of Theorem \ref{ramseyfilterequivalence}.
  
  Now assume that (b) holds. Thus fix some regular $\theta>\kappa$, and let (b) be witnessed by $\vec M$ and by $F$. Then $M=\bigcup_{i<\alpha}M_i\prec H(\theta)$ is a weak $\kappa$-model closed under ${<}\alpha$-sequences, $F$ is weakly amenable for $M$ and the ultrapower of $M$ by $F$ is well-founded. This shows that $\kappa$ is $\alpha$-Ramsey by Theorem \ref{modestequivalences}, (d).
  
  Assuming that (c) holds and that $\alpha>\omega$, the same argument shows that $\kappa$ is $\alpha$-Ramsey, this time making use of Theorem \ref{modestequivalences}, (e).
\end{proof}

\section{Absoluteness to L}\label{section:absoluteness}

Weakly Ramsey cardinals are downward absolute to $L$ by \cite[Theorem 3.12]{MR2830435}. Since $\omega_1$-Ramsey cardinals are Ramsey by Proposition \ref{w1RR}, they cannot exist in $L$. We want to show that $\omega$-Ramsey cardinals are downwards absolute to $L$. This proof is a variation of the proof of \cite[Theorem 3.4]{MR2830435}. We will make use of a slight adaption of what is known as the \emph{ancient Kunen lemma}.

\begin{lemma}\label{ancientkunen}
  Let $M\models\ZFC^-$, let $j\colon M\to N$ be an elementary embedding with critical point $\kappa$, such that $\kappa+1\subseteq M\subseteq N$. Assume that $|X|^M=\kappa$. Then $j\upharpoonright X\in N$.
\end{lemma}
\begin{proof}
  Note that $j\upharpoonright X$ is definable from an enumeration $f$ of $X$ in $M$ in order-type $\kappa$, together with $j(f)$, both of which are elements of $N$ by our assumptions. Namely, for $x\in X$,
\[j(x)=y\ \iff\ \exists\alpha<\kappa\ x=f(\alpha)\,\land\,y=j(f)(\alpha).\]
The lemma follows as $\kappa+1\subseteq N$ implies that this definition is absolute between $N$ and $V$.
\end{proof}


\begin{lemma}\label{SilverOmegaRamsey}
  If $0^\sharp$ exists, then all Silver indiscernibles are $\omega$-Ramsey in $L$.
\end{lemma}
\begin{proof}
  Let $I=\{i_\xi\mid\xi\in\On\}$ be the Silver indiscernibles, enumerated in increasing order. Fix a particular Silver indiscernible $\kappa$, let $\lambda=(\kappa^+)^L$, let $\theta=((2^\kappa)^+)^L$, and let $A$ be a subset of $\kappa$ in $L$. Define $j\colon I\to I$ by $j(i_\xi)=i_\xi$ for all $i_\xi<\kappa$ and $j(i_\xi)=i_{\xi+1}$ for all $i_\xi\ge\kappa$ in $I$. The map $j$ extends, via the Skolem functions, to an elementary embedding $j\colon L\to L$ with critical point $\kappa$. Let $U$ be the weakly amenable $L_\lambda$-normal filter on $\kappa$ generated by $j$. Since every $\alpha<\lambda$ has size $\kappa$ in $L_\lambda$, each $U\cap L_\alpha\in L_\lambda$ by weak amenability of $U$. Let $\langle M_i\mid i\in\omega\rangle$ be a sequence such that each $M_i\prec L_\theta$ is a weak $\kappa$-model in $L$, such that $A\in M_0$, and such that $M_i,U\cap M_i\in M_{i+1}$. For each $i<\omega$, let $j_i$ be the restriction of $j$ to $M_i$. Each $j_i\colon M_i\to j(M_i)$ has a domain of size $\kappa$ in $L_\theta$, and is hence an element of $L_{j(\theta)}\subseteq L$ by Lemma \ref{ancientkunen}.

To show that $\kappa$ is $\omega$-Ramsey in $L$, we need to construct in $L$, a weak $\kappa$-model $M^*\prec L_\theta$ containing $A$ as an element, and a $\kappa$-powerset preserving elementary embedding $j\colon M^*\to N^*$. Define in $L$, the tree $T$ of finite sequences of the form \[s=\langle h_0\colon M^*_0\to N^*_0,\ldots,h_n\colon M^*_n\to N^*_n\rangle\] ordered by extension and satisfying the following properties:
\begin{enumerate-(a)}
  \item $A\in M^*_0$, each $M^*_i\prec L_\theta$ is a weak $\kappa$-model,
  \item $h_i\colon M^*_i\to N^*_i$ is an elementary embedding with critical point $\kappa$,
  \item $N^*_i\subseteq L_{j(\theta)}$.
  \item[] Let $W_i$ be the $M^*_i$-normal filter on $\kappa$ generated by $h_i$.
  \item For $i<j\le n$, we have $M^*_i,W_i\in M^*_j$, $N^*_i\prec N^*_j$ and $h_j\supseteq h_i$.
\end{enumerate-(a)}
Consider the sequences $s_n=\langle j_0\colon M_0\to j(M_0),\ldots,j_n\colon M_n\to j(M_n)\rangle$. Each $s_n$ is clearly an element of $T$ and $\langle s_n\mid n\in\omega\rangle$ is a branch through $T$ in $V$. Hence the tree $T$ is ill-founded, and by absoluteness of this property, $T$ is ill-founded in $L$. Let $\langle h_i\colon M^*_i\to N^*_i\mid i\in\omega\rangle$ be a branch of $T$ in $L$, and let $W_i$ denote the $M^*_i$-normal filter on $\kappa$ induced by $h_i$. Let \[h=\bigcup_{i\in\omega}h_i,\ M^*=\bigcup_{i\in\omega}M^*_i\textrm{ and }N^*=\bigcup_{i\in\omega}N^*_i.\]
It is clear that $M^*\prec L_\theta$, $h\colon M^*\to N^*$ is an elementary embedding with critical point $\kappa$ and that $M^*$ is a weak $\kappa$-model containing $A$ as an element. If $x\subseteq\kappa$ in $N^*$, then $x=[f]_{W_i}\in N^*_i$ for some $i<\omega$ and some $f\colon\kappa\to M^*_i$ in $M^*_i$. But then $x=\{\alpha<\kappa\mid\{\beta<\kappa\mid\alpha\in f(\beta)\}\in W_i\}\in M^*_{i+1}\subseteq M^*$ by Property (4). This shows that $h$ is $\kappa$-powerset preserving and thus that $\kappa$ has the $\omega$-filter property in $L$, as desired.
\end{proof}

To show that $\omega$-Ramsey cardinals are downwards absolute to $L$, we need yet another characterization of $\omega$-Ramsey cardinals.

\begin{lemma}\label{wRcofomegaclub}
$\kappa$ is $\omega$-Ramsey if and only if for arbitrarily large regular cardinals $\theta$ and every subset $C$ of $\theta$, every $A\subseteq\kappa$ is contained, as an element, in some weak $\kappa$-model $M$ such that $\langle M,C\rangle\prec\langle H(\theta),C\rangle$, 
and for which there exists a $\kappa$-powerset preserving elementary embedding $j\colon M\to N$.
\end{lemma}
\begin{proof}
  The backward direction of the lemma is immediate. For the forward direction, assume that $\kappa$ is $\omega$-Ramsey, and let $\theta$ and $C$ be as in the statement of the lemma. By Corollary \ref{wramseyfilterequivalence}, $\kappa$ has the well-founded $\omega$-filter property. Now adapt the proof that the well-founded $\omega$-filter property implies $\omega$-Ramseyness, that is provided for Theorem \ref{ramseyfilterequivalence}. Namely, let the challenger simply play structures $M_\gamma$ which satisfy $\langle M_\gamma,C\rangle\prec\langle H(\theta),C\rangle$. Note that the resulting structure $M_\omega$ witnessing $\omega$-Ramseyness will satisfy $\langle M_\omega,C\rangle\prec\langle H(\theta),C\rangle$.
\end{proof}

We are finally ready to show that $\omega$-Ramsey cardinals are downwards absolute to $L$.

\begin{theorem}\label{wramseyinl}
  $\omega$-Ramsey cardinals are downwards absolute to $L$. 
\end{theorem} 
\begin{proof} 
  Let $\kappa$ be an $\omega$-Ramsey cardinal. By Lemma \ref{SilverOmegaRamsey}, we may assume that $0^\sharp$ does not exist, and thus that $(\kappa^+)^L=\kappa^+$ by a classic observation of Kunen for weakly compact cardinals (see e.g.\ \cite[Exercise 18.6]{MR1940513}). Fix $A\subseteq\kappa$ in $L$, and a regular cardinal $\theta\ge(2^\kappa)^+$. Let $C\subseteq\theta$ be the club of $\gamma<\theta$ for which $L_\gamma\prec L_\theta$. Using Lemma \ref{wRcofomegaclub}, we may pick a weak $\kappa$-model $M$ such that $\langle M,C\rangle\prec\langle  H(\theta),C\rangle$, containing $A$ as an element, with a $\kappa$-powerset preserving elementary embedding $j\colon M\to N$, such that $\cof(M\cap\kappa^+)=\omega$.

Let $\lambda=\kappa^+$ and let $\bar\lambda=M\cap\kappa^+=L^M\cap\kappa^+$ and note that $\cof(\bar\lambda)=\omega$ by the above. Restrict $j$ to $j\colon L^M\to L^N$. It is easy to see that $\kappa$-powerset preservation of the original embedding $j$ implies that $L_{\kappa^+}^M=L_{\kappa^+}^N$, and hence that the restricted embedding $j$ is again $\kappa$-powerset preserving. Moreover $L^M\prec L^{H(\theta)}=L_\theta=H(\theta)^L$. 

Let $U$ be the weakly amenable $L_\lambda^M$-normal filter on $\kappa$ generated by $j$. Since every $\alpha<\lambda$ in $M$ has size $\kappa$ in $L^M$, each $U\cap L_\alpha\in L_\lambda^M$ by weak amenability of $U$. Using that 
$L_\lambda^M=L_{\bar\lambda}$, construct a sequence $\langle M_i\mid i\in\omega\rangle$ 
such that each $M_i\prec L^M$ is a weak $\kappa$-model in $L^M$
, such that $A\in M_0$, and such that $M_i,U\cap M_i\in M_{i+1}$. Note that we can achieve $M_i\in L^M$ since $C$ is unbounded in $M\cap\theta$ by elementarity, by picking first -- externally -- a sufficiently large $\xi_i\in M\cap C$, and then picking $M_i\prec L_{\xi_i}^M$ in $L^M$ in each step $i$ of our construction.

For each $i<\omega$, let $j_i$ be the restriction of $j$ to $M_i$. Each $j_i\colon M_i\to j(M_i)$ has a domain of size $\kappa$ in $L^M$, and is hence an element of $L^N$ by Lemma \ref{ancientkunen}. Moreover since $L$ is $\Delta_1^{\ZF^-}$-definable, $L^N\subseteq L$, hence $j_i\in L$ for every $i<\omega$.

To show that $\kappa$ is $\omega$-Ramsey in $L$, we need to construct in $L$, a weak $\kappa$-model $M^*\prec L_\theta$ containing $A$ as an element, and a $\kappa$-powerset preserving elementary embedding $j\colon M^*\to N^*$. In order to do so, we now continue verbatim as in the proof of Lemma \ref{SilverOmegaRamsey}.
\end{proof} 


\section{The strategic filter property versus measurability}\label{section:measurability}

Note that we have not only introduced the $\gamma$-filter properties, but also the strategic $\gamma$-filter properties in Definition \ref{def:filterproperties}. While we have already provided a variety of results about the $\gamma$-filter properties, we do not know a lot about their strategic counterparts. However we want to close our paper with the following result, that was suggested to us by Joel Hamkins. We originally had a similar result, however with a much more complicated proof, starting from a much stronger large cardinal hypothesis. We would like to thank Joel Hamkins for letting us include his proof here.

\begin{definition}
  A cardinal $\kappa$ is \emph{$\lambda$-tall} if there is an embedding $j\colon V\to M$ with critical point $\kappa$ such that $j(\kappa)>\lambda$ and $M^\kappa\subseteq M$.
\end{definition}

\begin{proposition}[Hamkins]\label{separation}
 Starting from a $\kappa^{++}$-tall cardinal $\kappa$, it is consistent that there is a cardinal $\kappa$ with the strategic $\kappa^+$-filter property, however $\kappa$ is not measurable.
\end{proposition}
\begin{proof}
  By an unpublished result of Woodin (see \cite[Theorem 1.2]{MR2489293}), if $\kappa$ is $\kappa^{++}$-tall, then there is a forcing extension in which $\kappa$ is measurable and the $\GCH$ fails at $\kappa$ (this improves a classic result of Silver, where the same is shown under the assumption of a $\kappa^{++}$-supercompact cardinal). Now we may perform the standard reverse Easton iteration of length $\kappa$, to force the $\GCH$ below $\kappa$, in each step adding a Cohen subset to the least successor cardinal (of the current intermediate model) which has not been considered in the iteration so far. By the $\Pi^2_1$-indescribability of measurable cardinals, $\kappa$ can not be measurable in the resulting model, since if it were, the failure of the $\GCH$ at $\kappa$ would reflect below $\kappa$. But clearly, the measurability of $\kappa$ is resurrected after adding a Cohen subset to $\kappa^+$, by standard lifting arguments. 
  
\medskip

Assume that $\kappa$ is not measurable, but is so in a further $\Add{\kappa^+}{1}$-generic extension (we may assume this situation starting from a $\kappa^{++}$-tall cardinal by the above). Let $\dot U$ be an $\Add{\kappa^+}{1}$-name for a measurable filter on $\kappa$. Let $\theta>\kappa$ be a regular cardinal. We define a strategy for the judge in $G_{\kappa^+}^\theta(\kappa)$ as follows. Provided the challenger plays some $\kappa$-model $M_\alpha\prec H(\theta)$, the judge picks a condition $p_\alpha$ deciding $\dot U\cap\check M_\alpha=\check F_\alpha$, and then plays $F_\alpha$. She does this so that $\langle p_\alpha\mid\alpha<\kappa^+\rangle$ forms a decreasing sequence of conditions. Let $M$ denote the union of the models played by the challenger. $F=\bigcup_{\alpha<\kappa^+}F_\alpha$ is then an $M$-normal filter in the ground model. This shows that $\kappa$ has the strategic $\kappa^+$-filter property. 
\end{proof}

\section{Some Questions}\label{questions}

The following collection of questions originates from the first submitted version of our paper. Since we first circulated that version of our paper, most of these questions have been answered by very recent results of Victoria Gitman, Dan Nielsen and Philip Welch. Only Questions \ref{unanswered1} and \ref{unanswered2} remain unanswered. We would still like present our original questions in this section, including our remarks from before we learned about their answers. Those answers will then be presented in Section \ref{answers}.

\medskip

While for uncountable cardinals $\alpha$, we obtained a direct correspondence between $\alpha$-Ramseyness and the $\alpha$-filter property, the issue of potential ill-foundedness forced us to introduce the concept of the well-founded $\omega$-filter property, in order to characterize $\omega$-Ramseyness in terms of filter games. The following should have a negative answer.

\begin{question}\label{wfomegafilterproperty}
  Does the $\omega$-filter property imply the well-founded $\omega$-filter property?
\end{question}

We would expect the filter games $G_\gamma^\theta(\kappa)$ from Section \ref{filtergames} not to be determined in case $\gamma$ is an uncountable cardinal, 
and ask the following question, for which we expect a negative answer.

\begin{question}\label{questiondeterminacy} 
  If $\gamma$ is an uncountable cardinal and the challenger does not have a winning strategy in the game $G_\gamma^\theta(\kappa)$, does it follow that the judge has one?
\end{question}

Our definitions allow for many variations, some of which we have partially studied, and some of which we haven't yet looked at at all.

\begin{question}\label{unanswered1}
  What properties does one obtain by considering variants of the games $G_\gamma^\theta(\kappa)$, where rather than $M$-normal filters for $\kappa$-models $M\prec H(\theta)$, we consider either
\begin{itemize} 
\item
${<}\kappa$-complete filters on subsets of $\mathcal P(\kappa)$ of size $\kappa$,
\item
$M$-normal filters for arbitrary $\kappa$-models $M$, weak $\kappa$-models $M$, or
\item 
normal filters on subsets of $\mathcal P(\kappa)$ of size $\kappa$?
\end{itemize} 
\end{question}

We showed in Theorem \ref{wramseyinl} that $\omega$-Ramsey cardinals are downwards absolute to $L$, and a positive answer seems highly likely for the following.

\begin{question}\label{alphaRamseyDJ}
  If $\omega\le\alpha\le\kappa$, are $\alpha$-Ramsey cardinals downwards absolute to the Dodd-Jensen core model?
\end{question}

What is the relationship between $\omega$-Ramsey cardinals and other cardinals that are compatible with $L$? For example:

\begin{question} \label{2iterability}
  Does $2$-iterability imply $\omega$-Ramseyness, or conversely?
\end{question}

A direction of possible research that we have not looked into so far at all is the following.

\begin{question}\label{unanswered2} 
  The notions of Ramsey-like cardinals are connected to measurable cardinals in talking about filters on $\kappa$. Can we obtain interesting variants of other filter-based large cardinals, for example supercompact cardinals, in a similar way? Do they have similar connections to generalized filter games?
\end{question} 


Proposition \ref{separation} shows that the strategic $\kappa$-filter property does not imply that $\kappa$ is measurable, and we expect the following question to have a negative answer.

\begin{question}\label{strategicconsistencystrength}
Does $\kappa$ having the strategic $\kappa$-filter property have the consistency strength of a measurable cardinal?
\end{question}




\section{Final Remarks}\label{answers}
Many of our open questions have very recently been answered. Gitman showed that the well-founded $\omega$-filter property is strictly stronger than the $\omega$-filter property, thus answering our Question \ref{wfomegafilterproperty}. We would like to thank her for letting us include her proof here.

\begin{proposition}[Gitman]\label{victoriaomegafilterproperty}
  If $\kappa$ is $\omega$-Ramsey, then $\kappa$ is a limit of cardinals with the $\omega$-filter property.
\end{proposition}
\begin{proof} 
  Making use of the $\omega$-Ramseyness of $\kappa$, let $j\colon M\to N$ with $M\prec H(\theta)$ for some regular $\theta$ be induced by the weakly amenable $M$-normal filter $U$ on $\kappa$, so that $j$ is $\kappa$-powerset preserving. We want to argue that $\kappa$ has the $\omega$-filter property in $N$, thus yielding the statement of the proposition by elementarity. Fix a regular $N$-cardinal $\nu>\kappa$, and fix a strategy $\sigma\in N$ for the challenger, in the game $G_\omega^\theta(\kappa)$ of $N$. Consider the tree with nodes being finite sequences of valid moves in this game in which the challenger follows his strategy $\sigma$, with nodes ordered by end-extension. Using that $j$ is $\kappa$-powerset preserving, $U$ is weakly amenable for $N$, hence this tree has a branch in $V$, generated by the judge playing intersections of $U$ with the models played by the challenger. But by absoluteness of well-foundedness, this tree has a branch in $N$, yielding that there is a run of this game in $N$ which will be won by the judge, and hence that $\sigma$ is not a winning strategy for the challenger in $N$, as desired.
\end{proof}

Furthermore, Gitman showed that if $\kappa$ is $2$-iterable, then there is a proper class of $\omega$-Ramsey cardinals in $V_\kappa$, locating the consistency strength of $\omega$-Ramsey cardinals more finely in the large cardinal hierarchy, and essentially answering our Question \ref{2iterability}.

According to Nielsen, an easy adaption of arguments from \cite{MR2830435} shows that if there is no inner model of a strong cardinal, then for $\alpha\le\kappa$, $\alpha$-Ramsey cardinals are $\alpha$-Ramsey in $K$, answering our Question \ref{alphaRamseyDJ}.
Moreover, Nielsen and Welch have informed us that our filter games are closely related to the games $G_r(\kappa,\lambda)$ from \cite{MR2817562}, and that the strategic filter properties are closely related to the notions of \emph{very Ramseyness} from \cite{MR2817562}.
 
Moreover, Welch informed us that he showed that if the $\omega_1$-strategic filter property holds at a cardinal $\kappa$, and there is no inner model of a strong cardinal, then $\kappa$ is measurable in $K$. Thus in particular the existence of a cardinal with the $\omega_1$-strategic filter property is equiconsistent with the existence of a measurable cardinal. This provides a strong positive answer to Question \ref{strategicconsistencystrength}. It also yields an immediate negative answer to Question \ref{questiondeterminacy}. 

The results by Nielsen and Welch mentioned above are planned to be published in an upcoming paper of theirs.

\bibliographystyle{alpha}
\bibliography{class-forcing}

\end{document}